\documentclass{amsart}
\usepackage[nohug]{diagrams}
\usepackage[mathscr]{eucal}
\usepackage[pdftex]{graphicx}
\usepackage{amsfonts}
\usepackage{amsmath}
\usepackage{amsthm}
\usepackage{amssymb}
\usepackage{latexsym}
\newfont{\msam}{msam10}

\newtheorem{theorem}[]{Theorem}
\newtheorem{proposition}[]{Proposition}
\newtheorem{corollary}[]{Corollary}
\newtheorem{lemma}[]{Lemma}
\theoremstyle{definition}
\newtheorem{definition}[]{Definition}
\newtheorem{remark}[]{Remark}

\newtheorem{conjecture}[]{Conjecture}

\let\nc\newcommand

\nc{\la}{\label}

\def\bthm{\begin{theorem}}
\def\ethm{\end{theorem}}
\def\blemma{\begin{lemma}}
\def\elemma{\end{lemma}}
\def\bproof{\begin{proof}}
\def\eproof{\end{proof}}
\def\bprop{\begin{proposition}}
\def\eprop{\end{proposition}}
\def\bcor{\begin{corollary}}
\def\ecor{\end{corollary}}

\def\Gr{\mbox{\rm{Gr}}^{\mbox{\scriptsize{\rm{ad}}}}}
\def\Grad{\mbox{\tt Grad}}

\def\Z{\mathbb{Z}}

\def\O{\mathcal{O}}
\def\A{\mathbb{A}}
\def\K{\mathbb{K}}
\def\L{\mathscr{L}}

\def\GG{\mathfrak{G}}
\def\AG{\mathfrak{A}}
\def\BG{\mathfrak{B}}
\def\UG{\mathfrak{U}}

\def\D{\mathcal{D}}

\def\m{\mathfrak{m}}
\def\c{\mathbb{C}}

\def\CC{\mathcal{C}}

\def\PP{{\mathbb P}}

\nc{\Hom}{{\rm{Hom}}}
\nc{\Ext}{{\rm{Ext}}}
\nc{\HOM}{\underline{\rm{Hom}}}
\nc{\EXT}{\underline{\rm{Ext}}}
\nc{\TOR}{\underline{\rm{Tor}}}
\nc{\End}{{\rm{End}}}
\nc{\GL}{{\rm{GL}}}
\nc{\PGL}{{\rm{PGL}}}
\nc{\SL}{{\rm{SL}}}
\nc{\SB}{{\rm{SB}}}
\nc{\sll}{{\mathfrak{sl}}}
\nc{\Rep}{{\rm{Rep}}}
\nc{\ad}{{\rm{ad}}}
\nc{\dlim}{\varinjlim}

\newcommand{\Spec}{{\rm{Spec}}}

\newcommand{\Pic}{{\rm{Pic}}}
\newcommand{\Aut}{{\rm{Aut}}}

\newcommand{\into}{\,\,\hookrightarrow\,\,}

\begin{document}
\title[]{Trees, Amalgams and Calogero-Moser Spaces}
%
% \date{July 4, 2010}
%
\author{Yuri Berest}
\address{Department of Mathematics, Cornell University, Ithaca, NY 14853-4201, USA}
\email{berest@math.cornell.edu}
%
%\thanks{Berest's work partially supported by NSF grant DMS 09-01570.}
%
\author{Alimjon Eshmatov}
\address{Department of Mathematics, University of Arizona, Tucson, AZ 85721-0089 USA}
\email{alimjon@math.arizona.edu}
\author{Farkhod Eshmatov}
\address{Department of Mathematics, University of Michigan, Ann Arbor, MI 48109-1043, USA}
\email{eshmatov@umich.edu}
%
% \thanks will become a 1st page footnote.

\begin{abstract}
We describe the structure of the automorphism groups of algebras Morita equivalent to the first Weyl algebra
$ A_1 $. In particular, we give a geometric presentation for these groups in terms of amalgamated products,
using the Bass-Serre theory of groups acting on graphs. A key r\^ole in our approach is played by a
transitive action of the automorphism group of the free algebra $ \c \langle x, y\rangle $ on the
Calogero-Moser varieties $ \CC_n $ defined in \cite{BW}. Our results generalize well-known theorems of
Dixmier and Makar-Limanov on automorphisms of $ A_1 $, answering an old question of Stafford
(see \cite{St}). Finally, we propose a natural extension of the Dixmier Conjecture for $ A_1 $ to the class of
Morita equivalent algebras.
\end{abstract}
\maketitle
%
%\subsection*{Introduction}
%
 Let $\, A_1 := \c \langle x, \, y \rangle/(xy-yx-1) \,$
be the first Weyl algebra over $\c$ with canonical generators $x$ and $y$. In his classic paper \cite{D},
Dixmier described the group $\, \Aut \, A_1 \,$ of automorphisms of $ A_1 $: specifically, he proved that
$\, \Aut \, A_1 \,$ is generated by the following transformations
\begin{equation}
\label{Phs}
\Phi_p\,:\,(x,\,y) \mapsto (x,\,y + p(x))\ ,\qquad \Psi_q\,:\, (x,\,y) \mapsto (x + q(y),\,y)\ ,
\end{equation}
where $\,p(x) \in \c[x]\,$ and $\,q(y) \in \c[y]\,$. Using this result of Dixmier, Makar-Limanov (see \cite{ML1, ML2})
showed that $\, \Aut \, A_1 \,$ is isomorphic to the group $\, G_0 \subset \Aut \,\c \langle x, \, y \rangle \,$
of `symplectic' (i.e. preserving $\, \omega = xy-yx$) automorphisms of the free algebra
$\,\c \langle x, \, y \rangle \,$: the corresponding isomorphism
\begin{equation}
\label{isML}
G_0 \stackrel{\sim}{\to} \Aut \,A_1
\end{equation}
is induced by the canonical projection $\, \c \langle x, \, y \rangle \to A_1\,$. On the other hand, the results of
\cite{ML1} (see, e.g., \cite{C}) also imply that $ G_0 $ is given by the amalgamated free product
\begin{equation}
\label{Aut}
G_0 = A *_U B\ ,
\end{equation}
where $\, A \, $ is the subgroup of symplectic affine transformations
\begin{equation}
\label{A}
(x,\,y) \mapsto (ax+by+e,\,cx+dy+f)\ , \quad a,\,b, \ldots,
f \in \c\ ,\quad  ad-bc=1\ ,
\end{equation}
$\, B \, $ is the subgroup of triangular (Jonqui\`eres) transformations
\begin{equation}
\label{B}
(x,\,y) \mapsto (ax + q(y),\,a^{-1}y+h)\ , \quad a \in \c^*,\
h \in \c\ ,\quad  q(y) \in \c[y]\ ,
\end{equation}
and $\, U\,$ is the intersection of $ A $ and $ B $ in $ G_0 $:
\begin{equation}
\label{U}
(x,\,y) \mapsto (ax+by+e,\,a^{-1}y+h)\ , \quad a \in \c^*,\
b,\,e,\, h \in \c\ .
\end{equation}
Combining \eqref{isML} and \eqref{Aut}, we thus get decomposition $\,\Aut\,A_1 \cong A *_U B \,$,
which completely describes the structure of $\,\Aut \,A_1\,$ as a discrete group (cf. \cite{A}).

The aim of the present paper is to generalize the above results to the case when
$ A_1 $ is replaced by a noncommutative domain $D$, Morita equivalent to
$A_1$ as a $\c$-algebra. This question was originally posed by Stafford in \cite{St}
(see {\it loc. cit.}, p. 636). To explain why it is natural, we recall that the algebras $\,D\,$
are classified, up to isomorphism, by a single integer $\, n \ge 0 \,$; the corresponding isomorphism classes
are represented by the endomorphism rings $\,D_n := \End_{A_1} M_n \,$ of certain distinguished right ideals
of $ A_1 $ and can be realized geometrically as algebras of global differential operators on rational singular
curves (see \cite{K, BW1} and \cite{BW4} for a detailed exposition). Thus
the Dixmier group $\,\Aut \,A_1 = \Aut \,D_0 \,$ appears naturally as
the first member in the family $\,\{\Aut \,D_n\, :\, n \ge 0 \}\,$. Our aim is to describe the `higher' groups
in this family: in particular, to give a presentation of $\,\Aut \,D_n\,$ for arbitrary $ n \ge 0 $ in terms
of amalgamated products.

%\subsection*{Main Results}
The groups $ \Aut \,D_n $ for $ n \ge 1 $ can be naturally identified with subgroups of $\Aut \,D_0$.
To be precise, let $ \Pic\, D $ denote the (noncommutative) Picard group of a $\c$-algebra $D$. By definition,
$ \Pic\, D $ is the group of $\c$-linear Morita equivalences of the category of $ D$-modules; its elements can be
represented by the isomorphism classes of invertible $D$-bimodules $\,[P]\,$ (see, e.g., \cite{B}). There is a
natural group homomorphism $\,\omega_D:\, \Aut\,D \to \Pic\,D \,$, taking $\, \sigma \in \Aut\,D \,$ to the class of
the bimodule $ [{}_1 D_{\sigma}] $, and if $\, D' \,$ is a ring Morita equivalent to $ D $, with
progenerator $ M $, then there is a group isomorphism $\, \alpha_M:\, \Pic \,D' \stackrel{\sim}{\to} \Pic \,D\,$
given by $\,[P] \mapsto [M^* \otimes_D P \otimes _D M]\,$. Thus, in our situation, for each $ n \ge 0 $
we have the following diagram
\begin{equation}
\la{D1}
\begin{diagram}[small, tight]
\Aut  \,D_n    & \rTo^{\ \omega_{D_n}\ }  & \Pic \,D_n \\
\dDotsto^{i_n} &                      & \dTo_{\alpha_{M_n}} \\
\Aut \,D_0    & \rTo^{\ \omega_{D_0}\ }  & \Pic \,D_0 \\
\end{diagram}
\end{equation}
where the vertical map $ \alpha_{M_n} $ is an isomorphism and the two horizontal maps are injective.
Moreover, since $ D_0 = A_1 $, a theorem of Stafford (see \cite{St}) implies that $ \omega_{D_0} $ is actually
an isomorphism. Inverting this isomorphism, we define the embedding
$\, i_n:\,\Aut \,D_n \into \Aut \,D_0 \,$, which makes \eqref{D1} a commutative diagram.

Recall that we defined $ G_0 $ to be the automorphism group of the free algebra $\,\c \langle x, \, y \rangle\,$
preserving $ [x,\,y] $. Now, for $\, n > 0 \,$, we introduce the groups $ G_n $ geometrically,
in terms of a natural action of $ G_0 $ on the {\it Calogero-Moser spaces}\, (see \cite{W})
\begin{equation*}
\label{e1}
{\mathcal C}_n := \{\,(X,\,Y) \in \mbox{\tt Mat}_n(\c) \times
\mbox{\tt Mat}_n(\c)\ :\ \mbox{\tt rk}([X,\,Y] + I_n) = 1 \,\}/\,\mbox{\tt PGL}_n(\c)\ ,
\end{equation*}
where $ \mbox{\tt PGL}_n(\c) $ operates on matrices $ (X,\,Y) $ by simultaneous conjugation.
The action of $ G_0 $ on $ \CC_n $ is given by
\begin{equation}
\label{at}
(X,\,Y) \mapsto (\sigma^{-1}(X),\, \sigma^{-1}(Y))\ , \quad \sigma \in G_0\ ,
\end{equation}
where $ \sigma^{-1}(X) $ and $ \sigma^{-1}(Y) $ are the noncommutative polynomials
$\,\sigma^{-1}(x) \in \c \langle x, \, y \rangle \,$ and $\,\sigma^{-1}(y) \in \c \langle x, \, y \rangle \,$
evaluated at $(X,Y)$. It is known that $ {\mathcal C}_n $ is a smooth affine algebraic variety
of dimension $2n$, equipped with a natural symplectic structure, and it is easy to check that
$G_0$ preserves that structure. Now, a theorem of Wilson and the first author (see \cite{BW}) implies that
\eqref{at} is a transitive action for all $ n \ge 0 $.
We define the groups $ G_n $ to be the stabilizers of points of $ \CC_n $
under this action: precisely, for each $ n \ge 0 $, we fix a basepoint $ (X_0,\,Y_0) \in
{\mathcal C}_n $, with
\begin{equation*}
\label{base}
X_0 = \sum_{k=1}^{n-1} E_{k+1, k}
\quad , \qquad
Y_0 = \sum_{k=1}^{n-1} \,(k-n)\, E_{k, k+1}\ ,
\end{equation*}
where $ E_{i,j} $ stands for the elementary matrix with $(i,j)$-entry $1$, and let
\begin{equation*}
\label{gn}
G_n := \mbox{\tt Stab}_{G_0}(X_0, Y_0)\ ,\quad n \ge 0 \ .
\end{equation*}

The following result can be viewed as a generalization of the above-mentioned theorem
of Makar-Limanov; in a slightly different form, it has already appeared in
\cite{BW4} (cf. {\it loc. cit.}, p.~120; see also \cite{W2}).
\begin{theorem}
\label{T1}
There is a natural isomorphism of groups $\,G_n \stackrel{\sim}{\to} \Aut \, D_n\,$.
\end{theorem}
\noindent
Specifically, we have group homomorphisms
\begin{equation*}
\label{ison}
G_n \into G_0 \stackrel{\sim}{\to} \Aut \,A_1 \stackrel{i_n}{\hookleftarrow} \Aut \,\D_n\ ,
\end{equation*}
where the first map is the canonical inclusion, the second is the Makar-Limanov isomorphism
\eqref{isML} and $ i_n $ is the embedding defined by \eqref{D1}. We claim that the image
of $ i_n $ coincides with the image of $ G_n $, which gives the required isomorphism.

Theorem~\ref{T1} is a simple consequence of the main results of \cite{BW}: in fact, it is shown in \cite{BW}
that there is a natural $G_0$-equivariant bijection (called the Calogero-Moser correspondence)
between $\, \bigsqcup_{n \ge 0} {\mathcal C}_n \,$ and the space of isomorphism classes of right
ideals of $ A_1 $. Under this bijection, the points $ (X_0, Y_0) \in {\mathcal C}_n $
correspond precisely to the classes of the ideals $ M_n $.

We will use Theorem~\ref{T1} to give a geometric presentation for the groups $ \Aut \,D_n $.
To this end, we associate to each space $ \CC_n $ a graph $ \Gamma_n $ consisting of orbits of
certain subgroups of $ G_0 $ and identify $ G_n $ with the {\it fundamental group}
$\, \pi_1({\mathbf \Gamma}_n, \ast)\, $ of a graph of groups $ {\mathbf \Gamma}_n $ defined
by the stabilizers of points of those orbits in $ \Gamma_n $. The Bass-Serre theory of groups
acting on graphs \cite{Se} will give
then an explicit formula for $\, \pi_1({\mathbf \Gamma}_n, \ast)\, $ in terms of generalized
amalgamated products (see \eqref{fgr2} below).

To define the graph $ \Gamma_n $ we take the subgroups $A$, $B$ and $U$ of $G_0$
defined by the transformations \eqref{A}, \eqref{B} and \eqref{U}. Restricting the action of $ G_0 $ on $ \CC_n $
to these subgroups, we let $ \Gamma_n $ be the oriented bipartite graph, with vertex and edge sets
\begin{equation}
\label{gamman}
\mbox{\tt Vert}(\Gamma_{n}) := (A \backslash \mathcal{C}_n)\,\bigsqcup\, (B \backslash \mathcal{C}_n) \ ,\quad
\mbox{\tt Edge}(\Gamma_{n}) := U \backslash \mathcal{C}_n \ ,
\end{equation}
and the incidence maps $\,\mbox{\tt Edge}(\Gamma_{n}) \to \mbox{\tt Vert}(\Gamma_{n})\,$ given by the canonical
projections $\, i: U \backslash \mathcal{C}_n \to A \backslash \mathcal{C}_n \,$
and $\,\tau: U \backslash \mathcal{C}_n \to B \backslash \mathcal{C}_n  \,$. Since the elements of
$A$ and $B$ generate $ G_0 $ and $ G_0 $ acts transitively on each $ \CC_n $, the graph $ \Gamma_n $ is connected.

Now, on each orbit in $ A \backslash \mathcal{C}_n $ and $\,B \backslash \mathcal{C}_n\,$
we choose a basepoint and elements $\, \sigma_A \in G_0 \,$ and $\,\sigma_B \in G_0\, $ moving these basepoints
to the basepoint $ (X_0,\,Y_0) $ of $ {\mathcal C}_n $. Next, on each $U$-orbit $\,\O_U \in
U \backslash \mathcal{C}_n \,$ we also choose a basepoint and an element $\,\sigma_U \in G_0 \,$ moving this
basepoint to $ (X_0,\,Y_0) $ and such that $\,	\sigma_U \in \sigma_A A \,\cap \,\sigma_B B\,$,
where $ \sigma_A $ and $ \sigma_B $ correspond to the (unique) $A$- and $B$-orbits containing $ \O_U $.
Using a standard construction in the Bass-Serre theory (see \cite{Se}, Sect.~5.4), we then assign to the
vertices and edges of $ \Gamma_n $ the stabilizers $\,A_\sigma = G_n \cap \sigma A \sigma^{-1} \,$,
$\,B_\sigma = G_n \cap \sigma B \sigma^{-1} \,$, $\,U_\sigma = G_n \cap \sigma U \sigma^{-1} \,$  of the corresponding
elements $ \sigma $ in the graph of right cosets of $ G_0 $ under the action of $ G_n $.
These data together with natural group homomorphisms $\, a_\sigma: U_\sigma \hookrightarrow A_\sigma  \,$
and $\, b_\sigma:\,U_\sigma \hookrightarrow B_\sigma \,$ define a graph of groups $ {\mathbf \Gamma}_n $ over
$ \Gamma_n $, and its fundamental group $\, \pi_1({\mathbf \Gamma}_n,\,T)\,$
relative to a maximal tree $\,T \subseteq \Gamma_n \,$ has canonical presentation (see \cite{Se}, Sect.~5.1):
\begin{equation}
\label{fgr2}
\pi_1({\mathbf \Gamma}_n,\,T) = \frac{A_\sigma \ast_{\,U_\sigma} B_\sigma \ast \, \ldots\, \ast \langle\, \mbox{\tt Edge}(\Gamma_n \setminus T)\,\rangle}{(\,e^{-1} a_\sigma(g)\, e = b_{\sigma}(g)\, :\, \forall\,e \in
\mbox{\tt Edge}(\Gamma_n \setminus T),\, \forall\, g \in U_\sigma\,)}\ .
\end{equation}
In \eqref{fgr2}, the amalgams $\, A_\sigma \ast_{\,U_\sigma} B_\sigma \ast \, \ldots \,$ are taken along the stabilizers
of edges of the tree $ T $, while $\, \langle\, \mbox{\tt Edge}(\Gamma_n \setminus T)\,\rangle\,$ denotes
the free group based on the set of edges of $ \Gamma_n $ in the complement of $T$.

Our main observation is the following.
\begin{theorem}
\label{T2}
For each $ n \ge 0 $, the group $ G_n $ is isomorphic to $\, \pi_1({\mathbf \Gamma}_n,\,T)\,$. In particular,
$ G_n $ has an explicit presentation of the form \eqref{fgr2}.
\end{theorem}
\begin{proof}
One can prove Theorem~\ref{T2} using the standard Bass-Serre theory (as exposed in \cite{Se}, Ch~I, Sect.~5,
or \cite{DD}, Ch.~I, Sect.~9). However, it seems that the more economic and intuitively clearer proof is
based on topological arguments: namely, an abstract version of Van Kampen's Theorem, which we are now
going to explain.

Let $\,\GG_n := \CC_n \rtimes G_0 \,$ denote the (discrete) transformation groupoid
corresponding to the action of $ G_0 $ on $ \CC_n $. The canonical projection $\,p:\, \GG_n \to G_0 \,$ is then
a connected covering of groupoids\footnote{We refer to \cite{M}, Ch.~3, for the theory of coverings of groupoids.},
which maps identically the vertex group of $ \GG_n $ at $\, (X_0,\,Y_0) \in \CC_n \,$ to the subgroup
$ G_n \subseteq G_0 $. Now, each of the subgroups $A$, $B$ and $U$ of $ G_0 $ can be lifted
to $ \GG_n\,$: $\,p^{-1}(A) = \GG_n \times_{G_0} A \,$,  $\,p^{-1}(B) = \GG_n \times_{G_0} B \,$
and $\,p^{-1}(U) = \GG_n \times_{G_0} U \,$, and these fibred products
are naturally isomorphic to the subgroupoids $\, \AG_n := \CC_n \rtimes A $, $\, \BG_n := \CC_n \rtimes B \,$
and $\, \UG_n := \CC_n \rtimes U \,$ of $ \GG_n$, respectively. Since the coproducts in the category of
groups coincide with coproducts in the category of groupoids and the latter can be lifted through coverings
(see \cite{O}, Lemma~3.1.1), the decomposition \eqref{Aut} implies
\begin{equation}
\la{vank}
\GG_n = \AG_n *_{\UG_n} \BG_n\ ,\quad \forall\,n \ge 0\ .
\end{equation}
Note that, unlike $ \GG_n$, the groupoids $ \AG_n $, $ \BG_n $ and $ \UG_n $ are not transitive
(if $\, n \ge 1 $), so \eqref{vank} can be viewed as an analogue of the Seifert-Van Kampen Theorem
for non-connected spaces (see, e.g., \cite{Ge}, Ch.~6, Appendix). As in the topological situation, computing
the fundamental (vertex) group from \eqref{vank} amounts to contracting the connected components (orbits) of
$ \AG_n $ and $ \BG_n $ to points (vertices) and $ \UG_n $ to edges. This defines a graph which is exactly
$ \Gamma_n $. Now, choosing basepoints in each of the contracted components and assigning the fundamental groups
at these basepoints to the corresponding vertices and edges defines a graph of groups (see \cite{HMM}, p.~46).
By {\it loc.~cit.}, Theorem~3, this graph of groups is (conjugate) isomorphic to the graph $ {\mathbf \Gamma}_n $
described above, and our group $ G_n $ is isomorphic to $\, \pi_1({\mathbf \Gamma}_n,\,T)\,$.
\end{proof}
Theorems~\ref{T1} and~\ref{T2} reduce the problem of describing the groups
$ \Aut \,D_n  $ to a purely geometric problem of describing the structure of the orbit spaces
of  $A$ and $B$ and $U$ on the Calogero-Moser varieties $ {\mathcal C}_n$. Using the earlier results
of \cite{W} and \cite{BW} and some basic invariant theory, one can obtain much
information about these orbits (and thence about the groups $ G_n $). In particular, the graphs
$ \Gamma_n $ can be completely described for small $ n $; it turns that $ \Gamma_n $ is a finite tree for
$\,n = 0,\,1,\, 2\,$, but has infinitely many cycles for $ n \ge 3 $ (see examples below).

We now explain the origin of $ \Gamma_n $. It turns out that these graphs can be
realized as quotient graphs of a certain `universal' tree $ \Gamma $ on which all
the groups $ \Aut \,D_n $ naturally act. Our construction of $ \Gamma $ is motivated
by algebraic geometry: specifically, a known application of the
Bass-Serre theory in the theory of surfaces (see, e.g., \cite{GD}, \cite{Wr}). In that approach,
the automorphism group of an affine surface $ S $ is described via its action on a tree whose
vertices correspond to certain (admissible) projective compactifications of $ S $. Following
the standard (by now) philosophy in noncommutative geometry (see, e.g., \cite{SV}), we may
think of our algebra $ D $ as the coordinate ring of a `noncommutative affine surface';
a `projective compactification' of $ D $ is then determined by a choice of filtration. Thus, we will define
$ \Gamma $ by taking as its vertices a certain class of filtrations on the algebra $D$.
It turns out that these filtrations can be naturally parametrized by an infinite-dimensional
{\it adelic Grassmannian} $\,\Gr $ introduced in \cite{W1} and studied in \cite{W, BW, BW3}
(in particular, we rely heavily on results of \cite{BW3}). Our contruction is close in spirit to
Serre's classic application of Bruhat-Tits trees for computing arithmetic subgroups of $ \SL_2(\K) $
over the function fields of smooth curves (see \cite{Se}, Chap.~II, \S\,2); however, at the moment,
we are not aware of any direct connection.

We begin by briefly recalling the definition of $\, \Gr $.  Let $ \c[z] $ be the polynomial ring in one
variable $z$. For each $\, \lambda \in \c \,$, we choose a {\it $\lambda$-primary} subspace in $ \c[z] $,
that is, a $\c$-linear subspace $\, V_{\lambda} \subseteq \c[z] \,$ containing a power of the maximal ideal
$\,\m_\lambda \,$ at $\, \lambda $. We suppose that $\, V_{\lambda} = \c[z] \, $
for all but finitely many $\, \lambda$'s.  Let $\, V = \bigcap_{\lambda} V_{\lambda} \,$
(such a subspace $\, V \,$ is called {\it primary decomposable} in  $ \c[z] $) and, finally, let
$$
W = \prod_{\lambda}\, (z - \lambda)^{-n_{\lambda}} \, V \subset
\c(z) \, ,
$$
where $\, n_{\lambda} \,$ is the codimension of $\, V_{\lambda} \,$
in $\, \c[z] \,$. By definition, $\, \Gr $ consists of all subspaces
$\, W \subset \c(z) \,$ obtained in this way.  For each
$\, W \in \Gr $ we set
\begin{equation*}
\label{aw}
A_W := \{f \in \c[z] \,:\, f W \subseteq W \} \,.
\end{equation*}
Taking $\,\Spec\,$ of $\, A_W \,$ gives then a rational curve $X$,
the inclusion $\, A_W \into \c[z] \,$ corresponds to normalization
$\, \pi : \A_{\c}^1 \to X \,$ (which is set-theoretically a bijective map),
and the $\, A_W$-module $\, W \,$ defines
a rank 1 torsion-free coherent sheaf $\, \L \,$ over $\, X \,$. In this way,
the points of $\, \Gr $ correspond bijectively to isomorphism classes of triples
$\, (\pi, X, \L) \,$ (see \cite{W1}).

Now, following \cite{BW}, for $\, W \in \Gr $ we define\footnote{In geometric terms, $ D(W) $
can be thought of as the ring $ D_{\L}(X) $ of twisted differential operators on $X$ with
coefficients in $ \L $.}
\begin{equation}
\label{dw}
D(W) := \{\D \in \c(z)[\partial_z] \,:\, \D W \subseteq W \} \, ,
\end{equation}
where $\, \c(z)[\partial_z] \,$ is the ring of rational differential operators
in the variable $z$. This last ring carries two natural filtrations:
the standard filtration, in which both generators $ z $ and $ \partial_z $
have degree $1$, and  the differential filtration, in which $\,\deg(z) = 0\,$
and $\,\deg(\partial_z) = 1\,$. These filtrations induce two different filtrations
on the algebra $ D(W) $, which we denote by $ \{D_{\bullet}^A(W)\} $ and $ \{D_{\bullet}^B(W)\} $
respectively.

Now, let $D$ be a fixed domain Morita equivalent to $ A_1 $. Following \cite{BW3}, we consider
the set\footnote{More generally, we may think of $ \Gr $ as a groupoid, in which the objects are
the $W$'s and the arrows are given by the algebra isomorphisms $ D(W) \to D(W') $. For $ D = D(W) $,
the set $ \Gr(D) $ is then a costar in $ \Gr  $, consisting of all arrows with target at $W$.
In \cite{BW3}, this set was denoted by $\,\Grad\,D\,$.} $\, \Gr (D) \,$
of all algebra isomorphisms $\,\sigma_W: D(W) \to D \,$, where $ W \in \Gr $
(more precisely, $\, \Gr (D) \,$ is the set of all pairs $\, (W, \,\sigma_W) \,$, where
$\, W \in \Gr $ and $\, \sigma_W \,$ is an isomorphism as above). Each
$\, \sigma_W \in \Gr (D) $ maps the two distinguished filtrations $ \{D_{\bullet}^A(W)\} $ and
$ \{D_{\bullet}^B(W)\} $ into the algebra $ D $: we call their images the {\it admissible}
filtrations on $D$ of type $A$ and type $B$, respectively. Let $ \PP_{A}(D) $ and $ \PP_{B}(D) $ denote
the sets of all such filtrations coming from various $\,\sigma_W \in \Gr (D)\,$.
By definition, we have then two natural projections
\begin{equation}
\la{prc}
\PP_{A}(D)  \stackrel{\pi_A}{\longleftarrow} \Gr (D) \stackrel{\pi_B}{\longrightarrow} \PP_B(D)\ .
\end{equation}
We say that $\,(W, \sigma_W) \,$ and $\,(W', \sigma_W') \,$ are {\it equivalent} in $\, \Gr (D) \,$
if their images under $ \pi_A $ and $ \pi_B $ coincide. Writing $ \Gr (D)/\!\sim $ for the set of
equivalence classes under this relation, we define an oriented graph $ \Gamma $ by
$$
\mbox{\tt Vert}(\Gamma) := \PP_A(D) \,\bigsqcup\, \PP_B(D)\ ,  \quad \mbox{\tt Edge}(\Gamma) := \Gr (D)/\!\sim\ ,
$$
with incidence maps $\,\mbox{\tt Edge}(\Gamma) \to \mbox{\tt Vert}(\Gamma)\,$ induced by the
projections \eqref{prc}. Observe that the group $ \Aut \,D $ acts naturally on the set $\, \Gr (D)\,$
(by composition), and this action induces an action of $ \Aut \,D $ on the graph $ \Gamma $ via \eqref{prc}.
We write $\, \Aut \, D \backslash \Gamma\,$ for the corresponding quotient graph.

\begin{theorem}
\la{T3}

$(a)$ $ \Gamma $ is a tree, which is independent of $ D $ (up to isomorphism).

$(b)$ For each $ n \ge 0 $, the graph $ \Aut \, D_n \backslash \Gamma $ is naturally isomorphic to $ \Gamma_n $.

\end{theorem}

Theorem~\ref{T3} can be viewed as a generalization of the main results of \cite{BW3}.
Indeed, this last paper is concerned with a description of the maximal abelian ad-nilpotent ({\it mad})
subalgebras of $ D_n \,$: its main theorems (see {\it loc. cit.}, Theorem~1.5 and Theorem~1.6) say
that the space $ \mbox{\rm Mad}(D_n) $ of all mad subalgebras of $ D_n $ is independent of $ D_n $
and its quotient modulo the natural action of $ \Aut \,D_n $ is isomorphic to the orbit space
$ B \backslash \mathcal{C}_n $. Now, it is easy to
see that every mad sublagebra defines an admissible filtration on $ D_n $ of type $B$, and conversely
the zero degree component of every filtration of type $B$ is a mad subalgebra of $D_n$. Thus, we have
a natural bijection $\, \PP_B(D_n) \cong \mbox{\rm Mad}(D_n) $, which is equivariant under the action
of $ \Aut \,D_n $. This implies that $ \PP_B(D_n) $ does not depend on $ D_n $, which is part of
Theorem~\ref{T3}$(a)$, and
$$
\Aut \, D_n \backslash \PP_B(D_n)\,\cong \,\Aut \, D_n \backslash \mbox{\rm Mad}(D_n) \,\cong\,
B \backslash \mathcal{C}_n\ ,
$$
which is part of Theorem~\ref{T3}$(b)$. In fact, the entire Theorem~\ref{T3} can be proved using the techniques
of \cite{BW3}. We should also mention that for $ D = A_1(k) $ our construction of the tree $ \Gamma $
agrees with the one given in \cite{A}.

\vspace{1ex}

%\subsection*{Examples}
We now look at examples of the graphs $ \Gamma_n $ and groups $G_n$ for small $n$.
For $ n = 0 $, the space $ \CC_0 $ is just a point, and so are a fortiori
its orbit spaces. The graph $ \Gamma_0 $ is thus a segment, and the corresponding graph of
groups $ {\mathbf \Gamma}_0 $ is given by $\,[\,A \stackrel{U}{\longrightarrow} B\,] \,$.
Formula \eqref{fgr2} then says that $\,G_0 = A \ast_U B \,$, which agrees, of course, with
the Makar-Limanov isomorphism \eqref{Aut}.

For $ n =1 $, we have $\, {\mathcal C}_1 \cong \mathbb{C}^2 $, with $ (X_0, Y_0) $ corresponding
to the origin. Since each of the groups
$A$, $B$ and $U$ contains translations $\,(x,y) \mapsto (x+a, y+b) \,$, $\,a,b \in \c $,
they act transitively on $ {\mathcal C}_1 $. So again $ \Gamma_1 $ is just the segment, and
$ {\mathbf \Gamma}_1 $ is given by $\,[\,A_1 \stackrel{U_1}{\longrightarrow} B_1\,]\,$,
where $\, A_1 := G_1 \cap A \,$, $\, B_1 := G_1 \cap B \,$ and $\, U_1 := G_1 \cap U \,$.
Since, by definition, $ G_1 $ consists of all $\, \sigma \in G_0 \,$
preserving $\,(0, 0)\,$, the groups $A_1$, $ B_1 $ and $U_1 $ are obvious:
\begin{eqnarray}
A_1\!\!\!\!\! &:& (x,\,y) \mapsto (ax + by,\,cx+dy)\ , \quad a,\,b,\,c,\,d \in
\mathbb{C} \ ,\ ad-bc=1\ ,
\nonumber \\
B_1\!\!\!\!\! &:& (x,\,y) \mapsto (ax + q(y),\,a^{-1} y)\ , \quad a \in \mathbb{C}^*\ ,\
q \in \c[y]\ ,  \ q(0) = 0\ , \nonumber \\
U_1 \!\!\!\!\! &:& (x,\,y) \mapsto (ax + by,\,a^{-1} y)\ , \quad a \in \mathbb{C}^*\ ,\
b\in \mathbb{C}\ .
\nonumber
\end{eqnarray}
It follows from \eqref{fgr2} that $\, G_1 = A_1 \ast_{U_1} B_1 \,$.

For $ n = 2 $, the situation is already more interesting. A simple calculation shows that
$ U $ has three orbits in $ {\mathcal C}_2$: two closed orbits of dimension $3$ and one open orbit
of dimension $4$. Moreover, the $B$-orbits coincide with the $U$-orbits. Combinatorially, this means
that the group $A$ acts transitively, and the graph $ \Gamma_2 $ is a tree with one nonterminal and
three terminal vertices corresponding to the $A$-orbit and the $B$-orbits, respectively.
In this case, the graph of groups ${\mathbf \Gamma_2} $ is given by
$$
\begin{diagram}
&                        &                       &  G_{2,y} \rtimes \c^*    \\
&                        & \ruTo^{\c^*}     &               \\
&\c^*                   & \rTo^{\Z_{2}\quad} & G^{(1)}_{2, y} \rtimes \Z_{2} \\
&                         & \rdTo_{\c^*}     &               \\
&                         &                       & G_{2,x} \rtimes \c^*    \\
\end{diagram}
$$
where $ G_{2,x} $ and $ G_{2,y} $ are the subgroups of $ G_0 $ consisting of all transformations
$ \Phi_p $ and $ \Psi_q $ (see \eqref{Phs}),
with $ p \in \c[x] $ and $ q \in \c[y] $ satisfying $ p(0) = p'(0) = 0 $ and
$ q(0) = q'(0)= 0 $ respectively, and $\, G^{(1)}_{2, y} := \{\,\Phi_{-x}\,\Psi_q \,\Phi_{ x} \in G_0 \ :\ q \in \c[y]\ , \
q(\pm 1) = 0\,\}\,$. Formula \eqref{fgr2} yields the presentation
$$
\,G_2 = (G_{2,x} \rtimes \c^*) \ast_{\c^*}
(G_{2,y} \rtimes \c^*) \ast_{\Z_2} (G^{(1)}_{2, y}
\rtimes \Z_2) \, .
$$
In particular, $ G_2 $ is generated by its subgroups $\, G_{2,x} $, $\, G_{2,y} $, $\, G^{(1)}_{2, y}$
and $ \c^* $.

Now, for $ n = 3 $, the structure of the graph $ \Gamma_3 $ and the group $ G_3 $ is much more complicated.
The graph $ \Gamma_3 $ is  {\it not}
a tree: in fact, it has infinitely many circuits. Nevertheless, the group $ G_3 $ can still be described
explicitly:
$$
G_3 = \frac{\pi_1(T_3,\,G_3) \,\star\,
\langle E_+(\Gamma_3 \setminus
T_3)\rangle}{(\,e^{-1} \alpha_e(g)\, e = \alpha_{e^*}(g)\, :\, \forall\,e \in
E_+(\Gamma_3 \setminus T_3),\, \forall\, g \in G_e\,)} \ ,
$$
where $ T_3 $ is a (maximal) tree in $ \Gamma_3 $ given in Figure \ref{T3}, $\,\pi_1(T_3,\,G_3)\,$ is the
corresponding tree product of stabilizer groups, and the complement graph $\,\Gamma_3 \setminus T_3 \,$
is shown in Figure \ref{graph4}.

We would like to end this paper with some questions and conjectures.

\medskip

{1.} By \cite{ML1}, it is known that $ G_0 $ is isomorphic to the group $\,\mbox{\rm SAut}\,\A_{\c}^2 \,$ of
{\it symplectic} automorphisms of the affine plane $\,\A_{\c}^2\,$ (as in the case of the Weyl algebra,
the isomorphism $\,G_0 \cong \mbox{\rm SAut}\, \A_{\c}^2 \,$ is induced by the canonical projection
$\, \c\langle x,\,y \rangle \to \c[x,\,y] $). Thus, the groups $ G_n $ can be naturally identified
with subgroups of $\,\Aut\,\A_{\c}^2 \,$. Do these last subgroups have a geometric interpretation?

{2.}
In this paper, we have described the structure of $ G_n $ and $ \Aut \,D_n $ as discrete groups.
However, these two groups carry natural {\it algebraic} structures and can be viewed as
infinite-dimensional algebraic groups (in the sense of Shafarevich \cite{Sh}). Despite being isomorphic
to each other as discrete groups, they are not isomorphic as algebraic groups
(for $n=0$, this phenomenon was observed in \cite{BW}.) A natural question is to explicitly describe
the algebraic structures on $ G_n $ and $ \Aut \,D_n $; in particular, to compute the corresponding
(infinite-dimensional) Lie algebras. The last question was an original motivation for our work. For
$ G_0 $, the answer is known (see \cite{G}).

{3.}
Compute the homology of the groups $ G_n $ for all $n$. Again, for $ n = 0 \,$, the answer is known
(see \cite{Al}): $\,H_*(G_0,\,\Z) \cong H_*(\SL_2(\c),\,\Z)\,$. One may wonder whether the groups
$ H_*(G_n,\,\Z) $ are strong enough invariants to distinguish the algebras $ D_n $
up to isomorphism. Unfortunately, the answer is `no': in fact, it follows from our description of $ G_1 $ that
$\,H_*(G_1,\,\Z) \cong H_*(\SL_2(\c),\,\Z)\,$. However, for $ n \ge 2 $, it seems that the groups $ H_*(G_n,\,\Z) $
are neither isomorphic to $\,H_*(\SL_2(\c),\,\Z)\,$ nor to each other, so they may provide
interesting invariants.

{4.}
Finally, we would like to propose an extension of the well-known {\it Dixmier Conjecture}
for $ A_1 \,$ (see \cite{D}, Probl\`eme~11.1) to the class of Morita equivalent algebras.
We recall that if $ D $ is a domain Morita equivalent to $ A_1 $, then there is a unique
integer $ n \ge 0 $ such that $\,D \cong D_n \,$, where $\,D_n \,$ is the endomorphism ring of
the right ideal $\, M_n = x^{n} A_1 + (y+n x^{-1})\,A_1 $. For two unital $\c$-algebras $A$ and $B$,
we denote by $\,\Hom\, (A,\,B)\,$ the set of all unital $\c$-algebra homomorphisms $\,A \to B \,$.
\begin{conjecture}
\la{DC}
For all $\, n, m \ge 0 \,$, we have
\begin{equation*}
\Hom\, (D_n,\,D_m) = \left\{
\begin{array}{lll}
\emptyset & \mbox{if}\quad n \ne m \\*[1ex]
\Aut \,D_n & \mbox{if} \quad n = m
\end{array}
\right.
\end{equation*}
\end{conjecture}
\noindent
Formally, Conjecture~\ref{DC} is a strengthening of the Dixmier Conjecture for $ A_1 $: in fact,
in our notation, the latter says that $\,\Hom\, (D_0,\,D_0) = \Aut \,D_0\,$. Does actually the Dixmier
Conjecture imply Conjecture~\ref{DC}?

\subsection*{Acknowledgments}
We are grateful to J.~Alev, V.~Bavula, O.~Chalykh, K.~Vogtmann, D.~Wright, G.~Wilson and E.~Zelmanov for interesting discussions,
questions and comments. We would also like to thank D.~Wright for providing us with reference \cite{Wr},
which turned out to be very useful, and G.~Wilson for sending us a copy of Quillen's private notes on trees and amalgams.
This work was partially supported by NSF grant DMS 09-01570.

\bibliographystyle{amsalpha}

%\newpage
%
%\begin{figure}[htp]\begin{center}
%\hspace*{-1in}
%\includegraphics[viewport=50 120 700 650,clip]{graph3.pdf}
%\end{center}
%\caption[Fig]{The Graph of Orbits $\Gamma_3$}
%\label{graph3}
%\end{figure}
%
\newpage
\begin{figure}[htp]\begin{center}
\hspace*{-1in}
\includegraphics[trim = 10mm 80mm 15mm 45mm, clip, width=15cm]{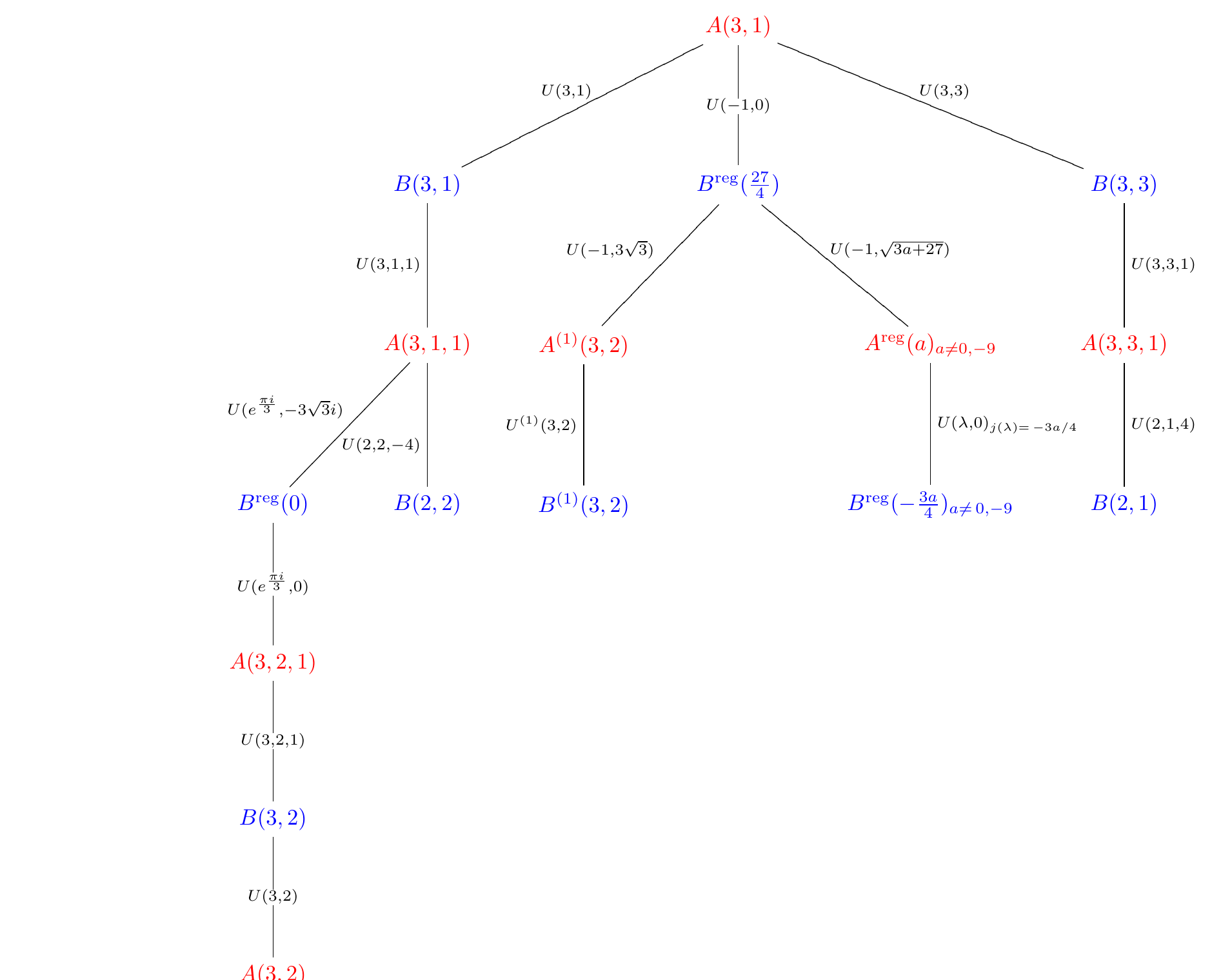}
\end{center}
\caption[Fig]{Maximal Tree $T_3$ (above); Graph $\Gamma_3 \backslash T_3$ (below)}
\label{graph4}
\begin{center}
\hspace*{-1in}
\includegraphics[trim = 10mm 100mm 20mm 45mm, clip, width=15cm]{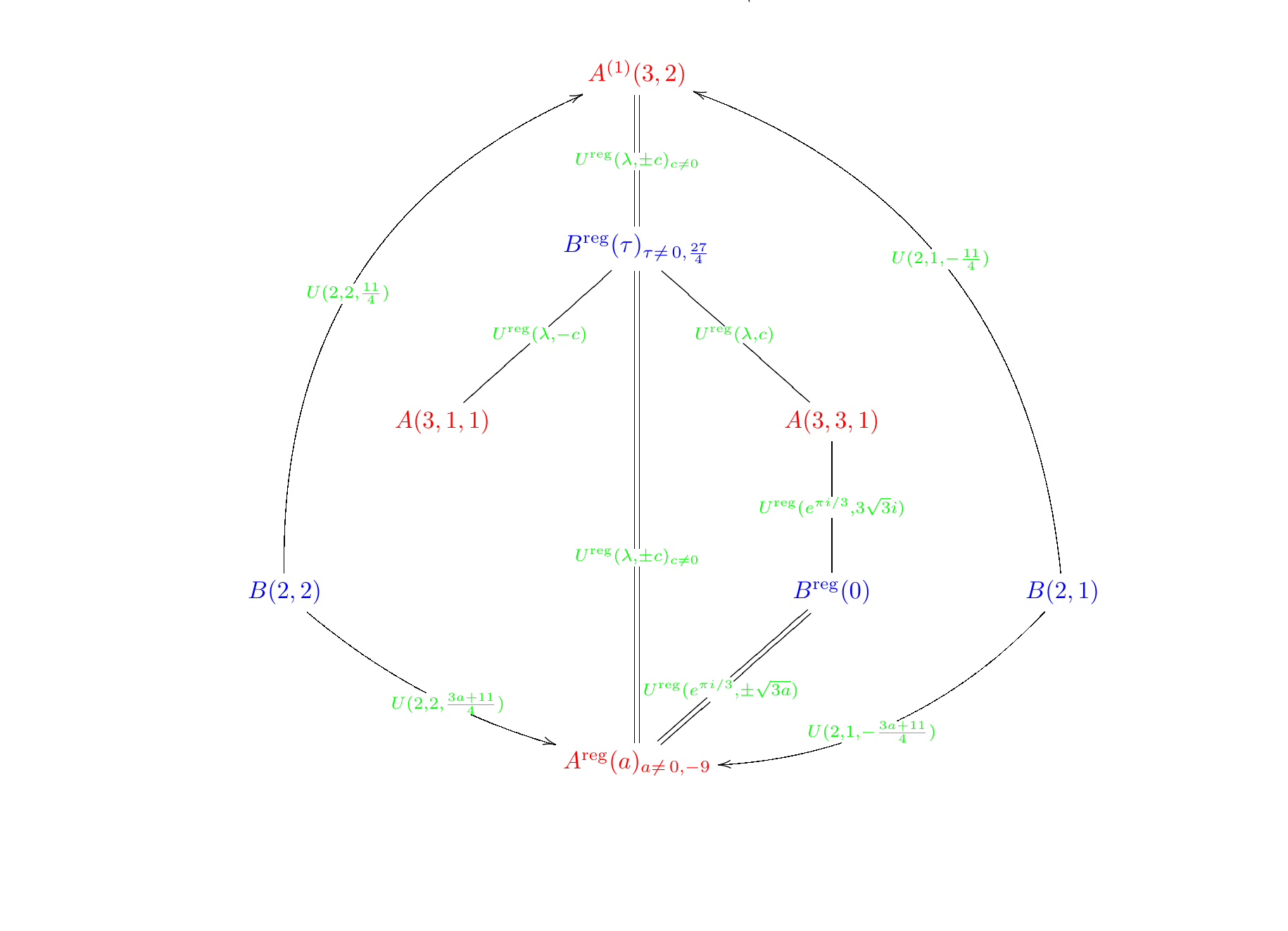}
\end{center}
\end{figure}
\newpage
\begin{figure}[htp]\begin{center}
\hspace*{-1in}
\includegraphics[trim = 10mm 10mm 20mm 20mm, clip, width=15cm]{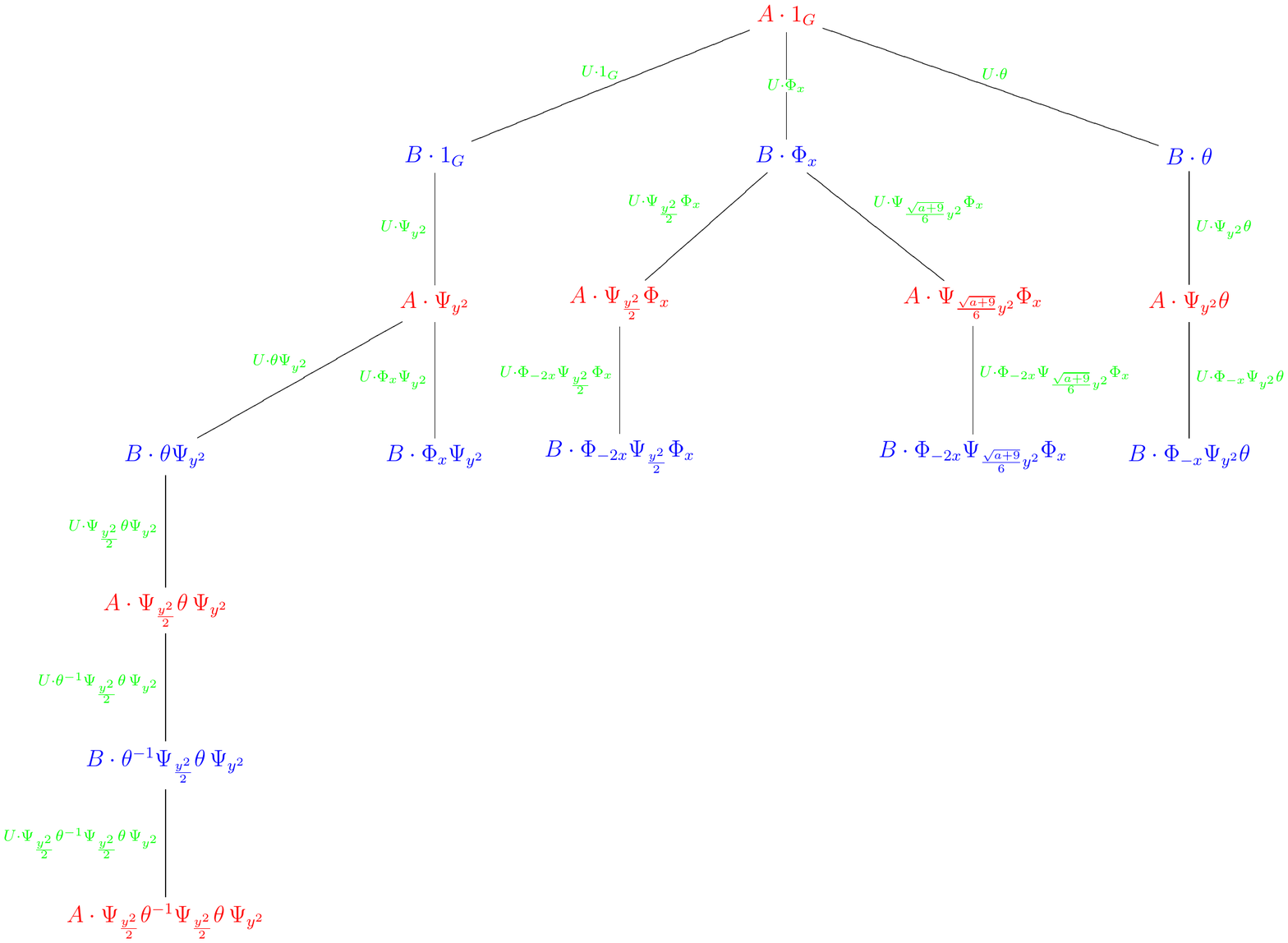}
\end{center}
\caption[Fig]{Maximal Tree of Groups $T_3$}
\label{T3}
%\end{figure}
%
 %\begin{figure}[htp]
 \begin{center}
\hspace*{-1in}
\includegraphics[trim = 10mm 10mm 15mm 40mm, clip, width=15cm]{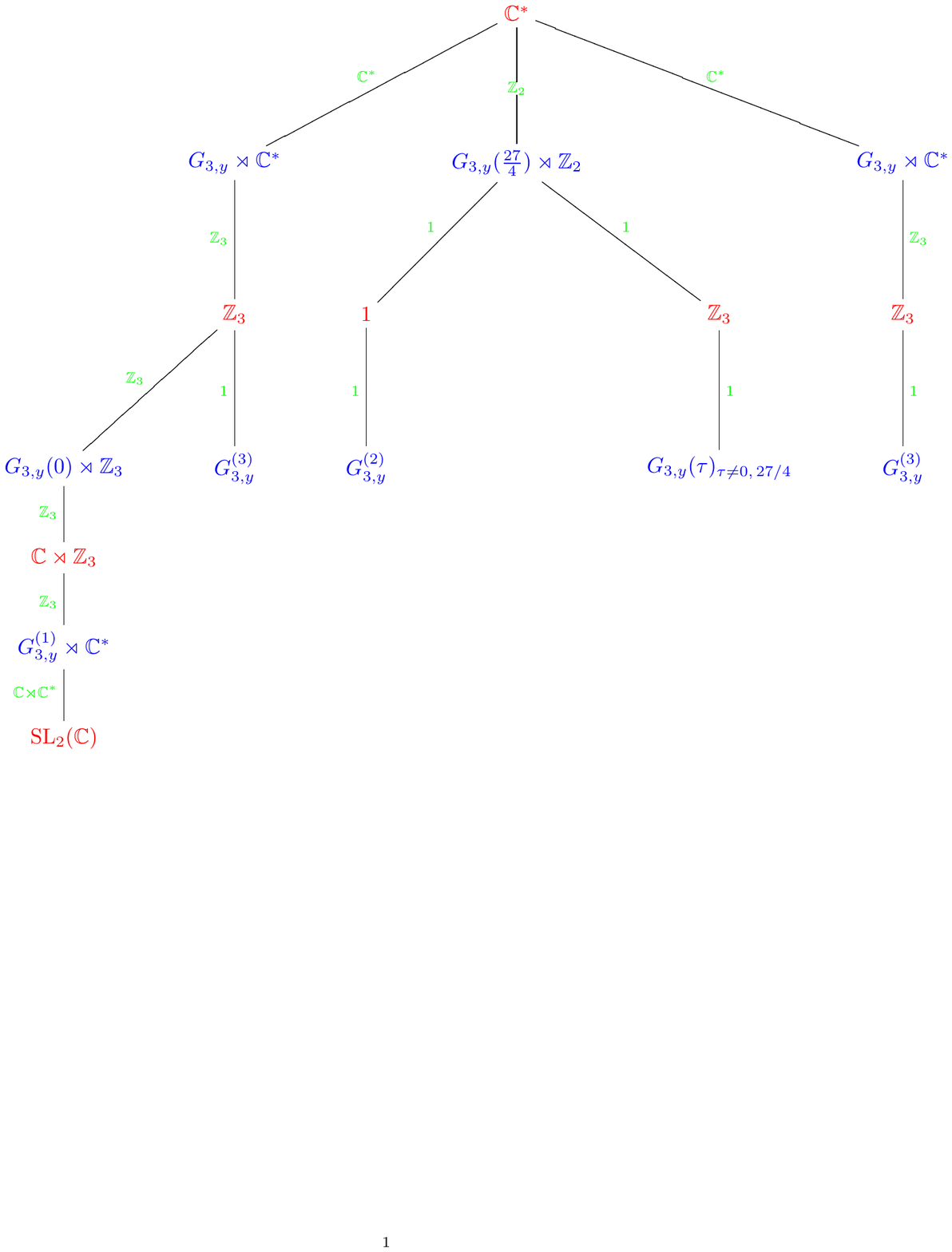}
\end{center}
\caption[Fig]{Graph ththt}
\end{figure}

\end{document}